\documentclass[12pt]{article}
\usepackage{amssymb, amsthm, amsmath, verbatim}
\newtheorem{theorem}{Theorem}
\newtheorem{lem}{Lemma}
\newtheorem{prop}{Proposition}

\author{A.I.~Nazarov, D.M.~Stolyarov, 
P.B.~Zatitskiy\footnote{A.N. was partially supported by grant NSh4210.2010.1. 
D.S. and P.Z. were supported by the Chebyshev Laboratory (Dept. of Mathematics and Mechanics, 
St.Petersburg State University) under the grant 11.G34.31.0026 of the Government of the Russian Federation.}}
\title{On the formula of regularized traces}

\begin{document}
\maketitle

We consider an operator $\mathbb L$ generated by a $2m$-order differential expression
\begin{equation}\label{operator}
ly\equiv(-1)^m D^{2m}y + \sum\limits_{k = 0}^{2m-2} p_k(x) D^ky
\end{equation}
and by $m$ boundary conditions
\begin{equation}\label{cond}
P_j(D)y(0)=0,\qquad j=1,\dots,m.
\end{equation}
Here $p_k\in L_{1,loc}(\mathbb R_+)$ are real functions while $P_j$ is a polynomial of degree $k_j$; 
moreover, the system of boundary conditions is assumed {\it normalized}, i.e. $0\leq k_1 < k_2 < ... < k_m \leq 2m-1$. 

Suppose this operator is self-adjoint in the space $L_2(\mathbb R_+)$, semibounded from below
and has purely discrete spectrum $\{\lambda_n\}_{n=1}^{+\infty}$ enumerated in ascending order according to the 
multiplicity.

Let $\mathbb Q$ be an operator of multiplication by a real function $q\in L_{\infty}(\mathbb R_+)$. Then
an operator $\mathbb L+\mathbb Q$ also has a purely discrete spectrum $\{\mu_n\}_{n=1}^{+\infty}$.\medskip

A number of papers beginning from the pioneering work \cite{L} are devoted to the calculation of spectral functions
and regularized traces of differential operators. The aim of our paper is to prove the following statement.

\begin{theorem}\label{main}
Let $q$ have a bounded support\footnote{This condition is used only to conclude the relation (\ref{Green}) from
Proposition \ref{Kost1}. It seems that this condition may be weakened. To do this one needs to
obtain a global estimate of the difference of spectral functions to operators $\mathbb L$ and $\widetilde{\mathbb L}$. 
However, up to the moment we do not know such an estimate.}, and the function $\psi(x)=\frac 1x\int\limits_0^x q(t)dt$
has a bounded variation at zero. Then the following relation holds:
\begin{equation}\label{prosto}
{\cal S}_1
\equiv \sum\limits_{n=1}^{\infty}\Big[\mu_n-\lambda_n-\frac{c_n}{\pi} \int\limits_{0}^{\infty}q(t) dt\Big]=-\psi(0+)\cdot\Big(\frac{m}{2}-\frac{1}{4}-\frac{\varkappa}{2m}\Big),
\end{equation}
{\it where }
$$c_1 = \lambda_{1}^{\frac{1}{2m}};\qquad c_n = \lambda_{n}^{\frac{1}{2m}}-\lambda_{n-1}^{\frac{1}{2m}},
\quad n>1;\qquad \varkappa=\sum_{j=1}^{m} k_j.$$ 
\end{theorem}

{\bf Remark}. Formula (\ref{prosto}) was conjectured by A.I.~Nazarov at the conference in Moscow, 2007, during the
A.S. Pechentsov lecture. For one-term boundary conditions $P_j(D)=D^{k_j}$ this formula was proved in the preprint
\cite{ZNS}\footnote{Three particular cases: 1) $k_j=2j-2$; 2) $k_j=2j-1$; 3) $k_j=j-1$ were considered earlier in
papers \cite{KP1}, \cite{SKP}, \cite{KP3}.}.\medskip

We introduce two auxiliary operators: 

$\widetilde{\mathbb L}$ is a $2m$-order operator with boundary conditions (\ref{cond}), self-adjoint in 
$L_2(\mathbb R_+)$; its lower-order coefficients are compactly supported and coincide with lower-order
coefficients in (\ref{operator}) on the segment $[0,R]\supset{\rm supp}(q)$;

${\mathbb L}_0$ is the operator $(-1)^m D^{2m}$ with boundary conditions (\ref{cond})\footnote{Note that this operator
does not need to be self-adjoint!}.\medskip

For $\lambda\in\mathbb R$ we denote by $\theta(x,y,\lambda)$ the spectral function of the operator $\mathbb L$, i.e.
the kernel of its spectral projector $E_{\lambda}$, see \cite{Go}. In a similar way, 
$\widetilde \theta(x,y,\lambda)$ is the spectral function of the operator $\widetilde{\mathbb L}$. Also we denote by
$H_0(x,y,\tau)$ the Green function of the operator ${\mathbb L}_0-\tau$ (for 
$\tau\not\in\overline{\mathbb R}_+$).\medskip

Let $\zeta$ be the value of $\tau^{\frac 1{2m}}$ such that $\arg(\zeta)\in[0,\frac {\pi}m[$, while
$z=\exp(\frac {i\pi}m)$. We introduce the matrix ${\cal B}(\zeta)=\big[P_j(iz^{\ell-1}\zeta)\big]_{\ell,j=1}^m$
and define $\Delta(\zeta)=\det({\cal B}(\zeta))$.\medskip

The following statements were proved in \cite{Ko1}, see also \cite{Ko2}. We state them with some redefinitions.

\begin{prop}\label{Kost1}
{\bf A} (\cite[Ch. 1, Sec. 2]{Ko1}). Let the following condition hold:\medskip

${\cal A}$. The matrix ${\cal B}(\zeta)$ is non-degenerate, and the elements of the inverse matrix satisfy the estimate
$\big[{\cal B}^{-1}(\zeta)\big]_{\ell j}= O(|\zeta|^{-k_j})$ as $|\zeta|\to\infty$.\medskip

Then the relation
$$\widetilde \theta(x,y,\lambda)=-\frac{1}{2\pi i}\int\limits_{|\tau|=\lambda}H_0(x,y,\tau)\,d\tau + 
O(\lambda^{-\frac 1{2m}}),\qquad \mbox{as}\quad \lambda\to+\infty,
$$
holds uniformly on $\overline{\mathbb R}_+^2$.\medskip

{\bf B} (\cite[Ch. 1, Sec. 4]{Ko1}). Let condition ${\cal A}$ be satisfied.
Then the relation
$$\theta(x,y,\lambda)=\widetilde \theta(x,y,\lambda)+o(1),\qquad \mbox{as}\quad \lambda\to+\infty,
$$
holds uniformly on $[0, R]^2$.
\end{prop}

First, we note that he assumptions of Proposition \ref{Kost1} may be weakened.

\begin{lem}\label{nocond}
For $|\zeta|$ large enough, condition $\cal A$ is always satisfied. Thus, this condition in Proposition
\ref{Kost1} can be omitted. 
\end{lem}

\begin{proof}
We claim that $\Delta(\zeta)$ is a $\varkappa$-degree polynomial of $\zeta$. Indeed, the columns of ${\cal B}(\zeta)$ 
are $k_j$-degree polynomials, while the highest degrees cannot reduce since the corresponding coefficient
is the Vandermonde determinant $\det({\mathbb W}(z^{k_1},\dots,z^{k_m}))\ne0$. Thus, $\Delta(\zeta)\ne0$ for
sufficiently large $|\zeta|$.

Further, the cofactor of $b_{\ell j}$ is evidently a polynomial of $\zeta$, and its degree does not exceed
$\varkappa-k_j$. The statement now follows from the Cramer formula.
\end{proof}

\begin{proof}[Proof of Theorem \ref{main}] Note that the left-hand side of (\ref{prosto}) can be rewritten as follows,
see~\cite[proof of Theorem 1]{KP3}:
$${\cal S}_1=\lim\limits_{\lambda\to+\infty}
\int\limits_0^{+\infty}q(x)\cdot\Big(\theta(x,x,\lambda)-\frac {\lambda^{\frac 1{2m}}}{\pi}\Big)\,dx.
$$
By Lemma \ref{nocond} this formula can be transformed to
\begin{equation}\label{Green}
{\cal S}_1=\lim\limits_{\lambda\to+\infty}\int\limits_0^{+\infty}q(x)\cdot
\Big(-\frac{1}{2\pi i}\int\limits_{|\tau|=\lambda}H_0(x,x,\tau)\,d\tau-\frac {\lambda^{\frac 1{2m}}}{\pi}\Big)\,dx.
\end{equation}

Now we use the explicit formula for $H_0$ obtained in \cite[Lemma 1]{KP1}. We state them with some redefinitions.
$$H_0(x,x,\tau)=\frac{i}{2m \zeta^{2m-1}}\cdot\sum\limits_{\alpha=1}^m z^{\alpha-1}\cdot
\Big(1-\frac 1{\Delta(\zeta)}\cdot {\sum\limits_{\beta=1}^m
\exp(i(z^{\alpha-1}+z^{\beta-1})\,x\zeta)\cdot\Delta_{\alpha\beta}(\zeta)}\Big),
$$
where the determinant $\Delta_{\alpha\beta}(\zeta)$ is obtained from $\Delta(\zeta)$ if we substitute
$P_j(-iz^{\alpha-1}\zeta)$ for $P_j(iz^{\beta-1}\zeta)$ in $\beta$-th line.
Surely, for this formula we need $\Delta(\zeta)\ne 0$ but this is the case for $|\tau|$ large enough.

We change the variable $\tau=\zeta^{2m}$ in the inner integral in (\ref{Green}). It is easy to see that
$$-\frac 1{2\pi i}\int\limits_{|\tau|=\lambda}\frac{i\,d\tau}{2m\zeta^{2m-1}}\sum\limits_{\alpha=1}^m z^{\alpha-1}
=-\int\limits_{\Gamma_\lambda}\frac {d\zeta}{\pi(1-z)}=\frac {\lambda^{\frac{1}{2m}}}{\pi}$$ 
(here $\Gamma_\lambda$ is the arc of the circle $\{\zeta=\lambda^{\frac{1}{2m}}e^{i\phi}:\phi \in ]0,\frac{\pi}{m}[ \}$).
Therefore (\ref{Green}) is rewritten as follows\footnote{In \cite[Theorem 2]{SKP} a formula similar to (\ref{nado})
was obtained under additional assumption: $\Delta(\zeta)\ne0$ for $\tau\in\mathbb R\setminus\{0\}$. This formula contains the real part of the integral over the segment $[0,\lambda^{\frac{1}{2m}}]$ instead of integral over
$\Gamma_\lambda$. In fact, this formula is correct only under a stronger assumption: $\Delta(\zeta)\ne0$ for $\tau\ne0$. In this case \cite[formula (0.5)]{SKP} is reduced to (\ref{nado}) by the Cauchy theorem.}:
\begin{equation}\label{nado}
{\cal S}_1=\frac{1}{2\pi}\lim\limits_{\lambda\to +\infty} \int\limits_0^{+\infty}q(x)\cdot\!\!
\int\limits_{\Gamma_\lambda}\sum\limits_{\alpha,\beta=1}^m
z^{\alpha-1}\exp(i(z^{\alpha-1}+z^{\beta-1})\,x\zeta)\cdot
\frac {\Delta_{\alpha\beta}(\zeta)}{\Delta(\zeta)}\ d\zeta\,dx.
\end{equation}

We denote ${\mathbb B}_{\alpha\beta}=\lim\limits_{\zeta\to\infty}\frac {\Delta_{\alpha\beta}(\zeta)}{\Delta(\zeta)}$ 
and claim that 
\begin{equation}\label{nado1}
{\cal S}_1=\frac{1}{2\pi}\lim\limits_{\lambda\to +\infty} \int\limits_0^{+\infty}q(x)\cdot\!\!
\int\limits_{\Gamma_\lambda}\sum\limits_{\alpha,\beta=1}^m
z^{\alpha-1}\exp(i(z^{\alpha-1}+z^{\beta-1})\,x\zeta)\cdot{\mathbb B}_{\alpha\beta}\ d\zeta\,dx.
\end{equation}
Indeed, since $\Delta_{\alpha\beta}$ is a polynomial of $\zeta$ and its degree does not exceed $\varkappa$,
the relation
\begin{equation}\label{BBB}
\frac {\Delta_{\alpha\beta}(\zeta)}{\Delta(\zeta)}={\mathbb B}_{\alpha\beta}+O(\zeta^{-1})
\end{equation}
holds as $|\zeta|\to\infty$. Further, for any $1\le\alpha,\beta \le m$ the inequality 
$0\le \arg(z^{\alpha-1}+z^{\beta-1})\le \frac{m-1}{m}\pi$ holds true. Therefore, the integrals
$${\cal I}_{\alpha\beta}(x)=\int\limits_{\Gamma_\lambda}z^{\alpha-1}\exp(i(z^{\alpha-1}+z^{\beta-1})\,x\zeta)\cdot
\Big(\frac {\Delta_{\alpha\beta}(\zeta)}{\Delta(\zeta)}-{\mathbb B}_{\alpha\beta}\Big)\ d\zeta
$$
are bounded by (\ref{BBB}) uniformly with respect to $x\ge0$ and tend to zero as $x>0$, $\lambda\to\infty$, by
the Jordan Lemma. By the Lebesgue Dominated Convergence Theorem, $\lim\limits_{\lambda\to +\infty}
\int\limits_0^{+\infty}q(x)\cdot{\cal I}_{\alpha\beta}(x)\,dx=0$, and (\ref{nado1}) follows.

We calculate the inner integral in (\ref{nado1}) and rewrite the formula for ${\cal S}_1$ as follows:
\begin{equation}\label{g}
{\cal S}_1=\frac 1{2\pi i}\cdot\lim_{s\rightarrow \infty} \int\limits_{0}^{\infty} q(x) g(sx) s\,dx,
\end{equation}
where
$$g(y)=\sum\limits_{\alpha,\beta=1}^m {\mathbb P}_{\beta\alpha}{\mathbb B}_{\alpha\beta}\ 
\frac{\exp(i(z^{\alpha}+z^{\beta})\,y)-\exp(i(z^{\alpha-1}+z^{\beta-1})\,y)}{y},
$$
while
$${\mathbb P}_{\beta\alpha}= \frac{z^{\alpha-1}}{z^{\alpha-1} + z^{\beta-1}}=\frac{1}{1+z^{\beta-\alpha}}.$$

To pass to the limit in (\ref{g}) in terms with $\alpha=\beta=1$ and $\alpha=\beta=m$, one uses the Riemann
localization principle, see, e.g., \cite[Ch. I, Sec. 33]{B}, and the Vall\`ee-Poussin test of convergence,
see \cite[Ch. III, Sec. 3]{B}. Other terms, after integration by parts, are covered by the Lebesgue Theorem, since
the exponents have negative real parts. Thus we arrive at\footnote{Note that passage from (\ref{g}) to 
(\ref{int}) follows also from \cite[Lemma 2]{KP3}.}
\begin{equation}\label{int}
{\cal S}_1=\frac 1 {2\pi i}\cdot\psi(0+)\int\limits_{0}^{\infty}g(y)\,dy=
-\frac{\psi(0+)}{2m}\cdot {\bf Sp}({\mathbb P}{\mathbb B})
\end{equation}
(the last equality follows from \cite[3.434.2]{GR}).\medskip

Note that the entries of matrix $\mathbb B$ are quotients of determinants composed of the leading coefficients
of polynomials that are entries of determinants $\Delta_{\alpha\beta}(\zeta)$ and $\Delta(\zeta)$. Direct calculation
via the Cramer formula gives
\begin{equation}\label{bb}
{\mathbb B}={\mathbb W}\cdot{\rm diag}[(-1)^{k_1},\dots,(-1)^{k_m}]\cdot {\mathbb W}^{-1}
\end{equation}
(we recall that ${\mathbb W}=({\mathbb W}_{\ell j})$ is the Vandermond matrix generated by the numbers $w_j=z^{k_j}$,
$j=1,\dots, m$).

Formula (\ref{bb}) shows that if $e^{(j)}=[1,w_j,\dots,w_j^{m-1}]^T$ is $j$-th column of matrix ${\mathbb W}$ then 
${\mathbb B}e^{(j)}=(-1)^{k_j}e^{(j)}$. Thus, vectors $e^{(j)}$ form the eigen basis of ${\mathbb B}$.\medskip

Next, by the geometrical progression sum formula, we rewrite the matrix ${\mathbb P}$ as follows:
\begin{equation}\label{progr}
{\mathbb P}=\lim\limits_{\rho \rightarrow 1-}\sum\limits_{n=0}^{\infty}(-1)^n \rho^n\,\varphi_n\overline{\varphi}^T_n,
\end{equation}
where $\varphi_n=(1,z^n,\dots,..,z^{(m-1)n})^T$.\medskip

Denote by $\mathbb K$ the set $\{k_j+2mp\,|\ 1\leq j\leq m, \ \ p\in\mathbb Z_+ \}$.

\begin{lem}\label{evenodd}
If $n\in \mathbb K$ then ${\bf Sp}(\varphi_n\overline{\varphi}^T_n{\mathbb B})=(-1)^n m$. 
If $n \notin \mathbb K$ then ${\bf Sp}(\varphi_n\overline{\varphi}^T_n{\mathbb B})=(-1)^{n+1}m.$ 
\end{lem}

To prove this lemma we need an obvious proposition.

\begin{prop}\label{hilb}
Let $u,v$ be elements of Hilbert space and let $a,b\in\mathbb C$. 
Then $(au+bv,u)=0$ implies $(au+bv,-au+bv)=|au+bv|^2$.
\end{prop}

\begin{proof}[Proof of Lemma \ref{evenodd}] First,
$${\bf Sp}(\varphi_n\overline{\varphi}^T_n{\mathbb B})= 
{\bf Sp}(\overline{\varphi}^T_n {\mathbb B}\varphi_n)=({\mathbb B}\varphi_n,\varphi_n).$$
We consider two subspaces
$$U=Lin\{e^{(j)}\,|\ k_j \equiv 1 \pmod 2\};\qquad V=Lin\{e^{(j)}\,|\ k_j \equiv 0 \pmod 2\}.$$ 

Note that $U\dotplus V=\mathbb C^m$, since the vectors $e^{(j)}$ form a basis.
Thus, there exists a decomposition $\varphi_n=u+v$ with $u \in U$, $v \in V$.
Further, we have shown that ${\mathbb B}\big|_V$ is an identity operator while ${\mathbb B}\big|_U$ 
multiplies by $-1$. This implies ${\mathbb B}\varphi_n=-u+v$.

{\bf 1}. Let $n \in \mathbb K$. Then $\varphi$ is an eigenvector of ${\mathbb B}$ corresponding to the eigenvalue
$(-1)^n$. Therefore $({\mathbb B}\varphi_n,\varphi_n)=(-1)^n (\varphi_n,\varphi_n)=(-1)^n m$.

{\bf 2}. Let $n \notin \mathbb K$ and $n \equiv 1 \pmod 2$. Then direct calculation shows that
$\varphi \perp U$. By Proposition \ref{hilb}, $({\mathbb B}\varphi_n,\varphi_n)=|\varphi_n|^2=m$.

{\bf 3}. In a similar way, if $n \notin \mathbb K$ and $n \equiv 0 \pmod 2$ then $\varphi \perp V$. 
By Proposition \ref{hilb}, $-({\mathbb B}\varphi_n,\varphi_n)=|\varphi_n|^2=m$.
\end{proof}

To complete the proof of Theorem we define the set $K=\{k_1, ..., k_m\}$. Then formula (\ref{progr}) and
Lemma \ref{evenodd} imply
\begin{multline*}
{\bf Sp}({\mathbb P}{\mathbb B})=\lim\limits_{\rho\rightarrow 1-} \sum \limits_{n=0}^\infty (-1)^n \rho^n 
{\bf Sp}(\varphi_n\overline{\varphi}^T_n{\mathbb B})=\\
=m\lim\limits_{\rho\rightarrow 1-}\Big( \sum \limits_{n\in \mathbb K} (-1)^n\rho^n(-1)^n + 
\sum \limits_{0\leq n\notin \mathbb K} (-1)^n \rho^n (-1)^{n+1}\Big)\\
=m\lim\limits_{\rho\rightarrow 1-}\Big(\sum \limits_{j\in K} \sum\limits_{p=0}^\infty \rho^{j+2mp} - 
\sum \limits_{0\leq j<2m,j\notin K}\sum\limits_{p=0}^\infty \rho^{j+2mp}\Big)\\
=m\lim\limits_{\rho\rightarrow 1-}\Big(\sum \limits_{j\in K} \frac{\rho^{j}}{1-\rho^{2m}} - \sum \limits_{0\leq j<2m,j\notin K}\frac{\rho^{j}}{1-\rho^{2m}}\Big).
\end{multline*}
Since the set $K$ contains exactly $m$ elements, we get
\begin{multline}\label{sled}
{\bf Sp}({\mathbb P}{\mathbb B})=m\lim\limits_{\rho\rightarrow 1-}\Big(\sum \limits_{j\in K}
\frac{\rho^{j}-1}{1-\rho^{2m}} - \sum \limits_{0\leq j<2m,j\notin K}\frac{\rho^{j}-1}{1-\rho^{2m}}\Big)=\\
=m\Big(\sum \limits_{j\in K}\frac{-j}{2m}-\sum \limits_{0\leq j<2m,j\notin K}\frac{-j}{2m}\Big) =
\frac{1}{2}\sum \limits_{0\leq j<2m}j - \sum \limits_{j\in K} j=\\
=\frac{m(2m-1)}{2}-\sum \limits_{j=1}^m k_j.
\end{multline}

Theorem follows immediately from (\ref{int}) and (\ref{sled}).
\end{proof}

We are grateful to I.A. Sheipak and A.A. Shkalikov who gave us the opportunity to learn the text of A.G.
Kostyuchenko's dissertation; to V.A. Kozlov and A.N. Podkorytov for useful discussions; to A.S. Pechentsov
and A.I. Kozko who provided us with the texts of papers \cite{KP3} and \cite{KP2}.

\end{document}